\documentclass{amsart}
\usepackage[all]{xy}
\usepackage{verbatim}
\usepackage{color}
\usepackage{amsthm}
\usepackage{amssymb}
\usepackage[colorlinks=true]{hyperref}

\numberwithin{equation}{section}

\newtheorem{theorem}[equation]{Theorem}
\newtheorem*{theorem*}{Theorem} \newtheorem{lemma}[equation]{Lemma}

\newtheorem*{conjecture*}{Mamma Conjecture}
\newtheorem*{conjecture1*}{Mamma Conjecture (revisited)}
\newtheorem{proposition}[equation]{Proposition}
\newtheorem{corollary}[equation]{Corollary}
\newtheorem*{corollary*}{Corollary}

\theoremstyle{remark}

\newtheorem{example}[equation]{Example}

\newtheorem{notation}[equation]{Notation}

\theoremstyle{remark}
\newtheorem{remark}[equation]{Remark}

\newcommand{\cA}{{\mathcal A}}
\newcommand{\cB}{{\mathcal B}}
\newcommand{\cC}{{\mathcal C}}
\newcommand{\cD}{{\mathcal D}}

\newcommand{\cN}{{\mathcal N}}
\newcommand{\cO}{{\mathcal O}}
\newcommand{\cP}{{\mathcal P}}

\newcommand{\cT}{{\mathcal T}}

\newcommand{\bbC}{\mathbb{C}}

\newcommand{\bbF}{\mathbb{F}}
\newcommand{\bbG}{\mathbb{G}}

\newcommand{\bbP}{\mathbb{P}}
\newcommand{\bbR}{\mathbb{R}}

\newcommand{\bbQ}{\mathbb{Q}}
\newcommand{\bbZ}{\mathbb{Z}}

\DeclareMathOperator{\NChow}{NChow} 
\DeclareMathOperator{\NNum}{NNum} 

\newcommand{\dgcat}{\mathrm{dgcat}} 

\newcommand{\perf}{\mathrm{perf}}

\newcommand{\dg}{\mathrm{dg}}

\newcommand{\Hom}{\mathrm{Hom}}
\newcommand{\End}{\mathrm{End}}

\newcommand{\op}{\mathrm{op}}

\newcommand{\too}{\longrightarrow}

\let\oldmarginpar\marginpar
\def\marginpar#1{\oldmarginpar{\tiny #1}}

\begin{document}

\title[A note on secondary $K$-theory II]{A note on secondary $K$-theory II}
\author{Gon{\c c}alo~Tabuada}

\address{Gon{\c c}alo Tabuada, Department of Mathematics, MIT, Cambridge, MA 02139, USA.}
\email{tabuada@math.mit.edu}
\urladdr{http://math.mit.edu/~tabuada}
\thanks{The author was partially supported by a NSF CAREER Award}

\date{\today}


\abstract{This note is the sequel to [{\em A note on secondary $K$-theory}. Algebra and Number Theory {\bf 10} (2016), no. 4, 887--906]. Making use of the recent theory of noncommutative motives, we prove that the canonical map from the derived Brauer group to the secondary Grothendieck ring has the following injectivity properties: in the case of a regular integral quasi-compact quasi-separated scheme, it is injective; in the case of an integral normal Noetherian scheme with a single isolated singularity, it distinguishes any two derived Brauer classes whose difference is of infinite order. As an application, we show that the canonical map is injective in the case of affine cones over smooth projective plane complex curves of degree $\geq 4$ as well as in the case of Mumford's (celebrated) singular surface.}}

\maketitle
\vskip-\baselineskip
\vskip-\baselineskip



\section{Introduction and statement of results}
A {\em differential graded (=dg) category $\cA$}, over a base commutative ring $k$, is a category enriched over complexes of $k$-modules; consult \cite{ICM-Keller}. Every (dg) $k$-algebra gives naturally rise to a dg category with a single object. Another source of examples is provided by schemes since the category of perfect complexes $\perf(X)$ of every quasi-compact quasi-separated $k$-scheme $X$ admits a canonical dg enhancement $\perf_\dg(X)$; see \cite[\S4.6]{ICM-Keller}. When $X$ is moreover quasi-projective this dg enhancement is unique; see \cite[Thm.~2.12]{LO}. Following \cite{Miami,finMot,IAS}, a dg category $\cA$ is called {\em smooth} if it is compact as a bimodule over itself and {\em proper} if all the complexes of $k$-modules $\cA(x,y)$ are compact in the derived category $\cD(k)$. Examples include the dg categories $\perf_\dg(X)$ associated to smooth proper $k$-schemes $X$.

Let $k$ be a base commutative ring and $X$ a quasi-compact quasi-separated $k$-scheme. Following \cite[\S2.2]{Azumaya}, resp. \cite{Toen1}\cite[\S4.4]{Toen2}, the derived Brauer group $\mathrm{dBr}(X)$, resp. secondary Grothendieck ring $K_0^{(2)}(X)$, is defined as the set of Morita equivalence classes of sheaves of dg Azumaya algebras\footnote{Consult \cite[Appendix B]{Separable} for several properties of these sheaves of dg Azumaya algebras.} $A$, resp. as the quotient of~the~free abelian group on the Morita equivalence classes of sheaves of smooth proper dg categories $\cA$ by the relations $[\cB]=[\cA]+[\cC]$ arising from short exact sequences $0\to \cA \to \cB \to \cC \to 0$. The group structure of $\mathrm{dBr}(X)$, resp. the multiplication law of $K_0^{(2)}(X)$, is induced by the derived tensor product and the inverse of $A$, resp. $\cA$, is given by $A^\op$, resp. $\cA^\op$. Since a sheaf of dg Azumaya algebras is, in particular, a sheaf of smooth proper dg categories, we have a canonical map:
\begin{equation}\label{eq:canonical-1}
\mathrm{dBr}(X) \too K_0^{(2)}(X)\,.
\end{equation}
The canonical map \eqref{eq:canonical-1} may be understood as the ``categorification'' of the canonical map from the Picard group $\mathrm{Pic}(X)$ to the Grothendieck ring $K_0(X)$. However, in contrast with $\mathrm{Pic}(X) \to K_0(X)$, \eqref{eq:canonical-1} does not seem to admit a ``determinant'' map in the converse direction. This makes the study of the injectivity of the canonical map \eqref{eq:canonical-1} a rather difficult problem.

As proved in \cite[Cor.~3.8]{Azumaya} (see also \cite[Appendix~B]{Separable}), we have an isomorphism:
\begin{equation}\label{eq:bijection-key2}
\mathrm{dBr}(X) \simeq H^1_{\mathrm{et}}(X,\bbZ) \times H^2_{\mathrm{et}}(X,\bbG_m)\,.
\end{equation}
Via \eqref{eq:bijection-key2}, the Brauer group $\mathrm{Br}(X)$ corresponds to a subgroup of the cohomological Brauer group $\mathrm{Br}'(X):=H^2_{\mathrm{et}}(X,\bbG_m)_{\mathrm{tor}}$. Therefore, given any non-torsion class $\alpha \in H^2_{\mathrm{et}}(X, \bbG_m)$, there exists a sheaf of dg Azumaya algebras $A_\alpha$, representing $\alpha$, which is {\em not} Morita equivalent to a sheaf of ordinary Azumaya algebras. Recall from \cite[Cor.~7.9.1]{Pic-Weibel} that the group $H^1_{\mathrm{et}}(X,\bbZ)$ is torsion-free and that it vanishes whenever $X$ is normal. When $X$ is regular integral, we have $H^1_{\mathrm{et}}(X,\bbZ)=0$ and every class of $H^2_{\mathrm{et}}(X,\bbG_m)$ is torsion; see \cite[Prop.~1.4]{Grothendieck2}. Hence, in this latter case, $\mathrm{dBr}(X)$ reduces to $\mathrm{Br}'(X)$. Finally, 
whenever $X$ admits an ample line bundle, $\mathrm{Br}'(X)$ reduces\footnote{As proved in \cite[\S IV Prop.~2.15]{Milne}, the reduction of the cohomological Brauer group $\mathrm{Br}'(X)$ to the Brauer group $\mathrm{Br}(X)$ holds always locally in the Zariski topology.} to the Brauer group $\mathrm{Br}(X)$;~see~\cite{deJong}.

Our results are divided in two parts: in the first part we consider the case where $X$ is regular and in the second part we allow the existence of an isolated singularity.
\subsection*{First part}
Our first main result is the following:
\begin{theorem}\label{thm:main3}
When $X$ is a regular integral quasi-compact quasi-separated $k$-scheme, the canonical map \eqref{eq:canonical-1} is injective.
\end{theorem}
Theorem \ref{thm:main3} applies, in particular, to the following families of schemes:
\begin{example}[Toric varieties]
Let $k$ be an algebraically closed field of~characteristic zero, $n>0$ an integer, $\Delta$ a finite fan on $\bbR^n$, and $X$ the~associated toric $k$-variety. Following \cite[Thm.~1.1]{Ford}, whenever $X$ is smooth, we have $\mathrm{Br}'(X)\simeq \mathrm{Br}(X)\simeq \oplus^{n-1}_{i=1} \Hom(\bbZ/a_i,\bbQ/\bbZ)^{n-i}$, where $a_1, \ldots, a_n$ are the invariant factors of $\Delta$.
\end{example}
\begin{example}[Complex surfaces]
Let $X$ be a complex simply-connected smooth projective surface. Following \cite[\S5.4]{Caldararu}, we have $\mathrm{Br}'(X)\simeq \mathrm{Br}(X)\simeq \Hom(T(X),\bbQ/\bbZ)$, where $T(X)$ denotes the transcendental lattice $\mathrm{NS}(X)^\perp \subset H^2(X,\bbZ)$. In the case of a $K3$ surface, the preceding computation reduces to $\mathrm{Br}(X)\simeq (\bbQ/\bbZ)^{\oplus 22- \rho(X)}$, where $\rho(X)$ denotes the Picard number of $X$. 
\end{example}
\begin{example}[Calabi-Yau varieties]
Let $X$ be a complex smooth projective Calabi-Yau $n$-fold. Whenever $n\geq 3$, we have $\mathrm{Br}'(X)\simeq \mathrm{Br}(X)\simeq H^3(X,\bbZ)_{\mathrm{tor}}$. Consult \cite{Addington, Gross, Schnell} for the construction of derived Morita equivalent complex Calabi-Yau $3$-folds $X$ and $Y$ such that $\mathrm{Br}(X)\simeq \bbZ/8\times \bbZ/8$ and $\mathrm{Br}(Y)=0$. This shows, in particular, that the derived Brauer group is {\em not} derived Morita invariant.
\end{example}
\begin{example}[Rational varieties]
Let $k$ be a field and $X$ a smooth projective rational $k$-variety. Following \cite[\S7]{Grothendieck3}, we have $\mathrm{Br}'(X)\simeq\mathrm{Br}(X)\simeq\mathrm{Br}(k)$.
\end{example}
By construction, every morphism $f\colon X \to Y$ between quasi-compact quasi-separated $k$-schemes gives rise to the following commutative square:
\begin{equation*}
\xymatrix{
\mathrm{dBr}(Y) \ar[d]_-{f^\ast} \ar[r]^-{\eqref{eq:canonical-1}} & K_0^{(2)}(Y) \ar[d]^-{f^\ast} \\
\mathrm{dBr}(X) \ar[r]_-{\eqref{eq:canonical-1}} & K_0^{(2)}(X)\,.
}
\end{equation*}
\begin{corollary}\label{cor:main1}
Let $[A], [B] \in \mathrm{dBr}(Y)$. If there exists a morphism $f\colon X \to Y$, with $X$ as in Theorem \ref{thm:main3}, such that $[f^\ast(A)]\neq[f^\ast(B)]$ in $\mathrm{dBr}(X)$, then the images of $[A]$ and $[B]$ under the canonical map \eqref{eq:canonical-1} are different.
\end{corollary}
\begin{example}[Henselian local rings]\label{ex:localII}
Let $S$ be an henselian local $k$-algebra with residue field $k'$. As proved in \cite[\S IV Cor.~2.13]{Milne}, the assignment $A \mapsto A \otimes^{\bf L}_S k'$ gives rise to a group isomorphism $\mathrm{Br}(S) \simeq \mathrm{Br}(k')$. Moreover, $H^1_{\mathrm{et}}(\mathrm{Spec}(S),\bbZ)=0$ and $\mathrm{Br}(S)\simeq H^2_{\mathrm{et}}(\mathrm{Spec}(S),\bbG_m)$; see \cite[Thm.~2.5]{Pic-Weibel} and \cite[Cor.~2.5]{Grothendieck}, respectively. Therefore, Corollary \ref{cor:main1} (with $X=\mathrm{Spec}(k')$ and $Y=\mathrm{Spec}(S)$) implies that the canonical map $\mathrm{dBr}(Y) \to K_0^{(2)}(Y)$ is injective.
\end{example}
\begin{remark}[Previous results from \cite{ANT}]
In the particular case where $k$ is a field\footnote{When $k$ is a field, every dg Azumaya $k$-algebra is Morita equivalent to an ordinary Azumaya $k$-algebra (=central simple $k$-algebra); see \cite[Prop.~2.12]{Azumaya}.} and $X=\mathrm{Spec}(k)$, let us write the canonical map \eqref{eq:canonical-1} as follows:
\begin{equation}\label{eq:canonical-new}
\mathrm{Br}(k) \too K_0^{(2)}(k)\,.
\end{equation}
The main results of \cite{ANT} are the following:
\begin{itemize}
\item[(a)] When $\mathrm{char}(k)=0$, the canonical map \eqref{eq:canonical-new} is injective; see \cite[Thm.~2.5]{ANT}.
\item[(b)] When $\mathrm{char}(k)=p>0$, the restriction of the canonical map \eqref{eq:canonical-new} to the $p$-primary torsion subgroup $\mathrm{Br}(k)\{p\}$ is injective; see \cite[Thm.~2.7 and Cor.~2.8]{ANT}.
\end{itemize}
The proofs of items (a)-(b) are different. Moreover, they both use in an essential way the semi-simplicity property of the category of noncommutative numerical motives. In \S\ref{sec:proof1} below we prove that the canonical map \eqref{eq:canonical-new} is always injective! Our new proof works in arbitrary characteristic and does not uses the semi-simplicity property of the category of noncommutative numerical motives. 

The proof of Theorem \ref{thm:main3} is obtained by combining the injectivity of \eqref{eq:canonical-new} with the injectivity of the restriction homomorphism $H^2_{\mathrm{et}}(X;\bbG_m) \to H^2_{\mathrm{et}}(k(X),\bbG_m)$, where $k(X)$ stands for the function field of $X$; consult \S\ref{sec:proof1} for details. 
\end{remark}

\subsection*{Second part}
Our second main result is the following:
\begin{theorem}\label{thm:five}
Let $X$ be an integral normal Noetherian\footnote{Recall that a Noetherian scheme is quasi-compact and quasi-separated.} $k$-scheme, which admits an ample line bundle, with a single isolated singularity. Given sheaves of dg Azumaya algebras $A$ and $B$ such that $A^\op \otimes^{\bf L} B$ is not Morita equivalent to a sheaf of ordinary Azumaya algebras, the images of $[A]$ and $[B]$ under the canonical map \eqref{eq:canonical-1} are different.
\end{theorem}
Intuitively speaking, Theorem \ref{thm:five} shows that the canonical map \eqref{eq:canonical-1} distinguishes any two derived Brauer classes whose difference is of infinite order.
\begin{corollary}\label{cor:last}
Given a sheaf of dg Azumaya algebras $A$ which is not Morita equivalent to a sheaf of ordinary Azumaya algebras, the restriction of the canonical map \eqref{eq:canonical-1} to the subgroup $\bbZ$ of $\mathrm{dBr}(X)$ generated by $[A]$ is injective.
\end{corollary}
\begin{corollary}\label{cor:(non)torsion}
Given a sheaf of dg Azumaya algebras $A$ which is not Morita equivalent to a sheaf of ordinary Azumaya algebras, the image of $[A]$ under \eqref{eq:canonical-1} is different from the images of the sheaves of ordinary Azumaya algebras.
\end{corollary}
\begin{proof}
Combine isomorphism \eqref{eq:bijection-key2} and Theorem \ref{thm:five} (with $B$ a sheaf of ordinary Azumaya algebras) with the fact that in every abelian group the product of a non-torsion element with a torsion element is always a non-torsion element.
\end{proof}
Corollary \ref{cor:(non)torsion} shows that \eqref{eq:canonical-1} distinguishes torsion from non-torsion classes. 
\begin{corollary}\label{cor:new}
Let $X$ be as in Theorem \ref{thm:five}. Whenever the kernel of the restriction homomorphism $H^2_{\mathrm{et}}(X,\bbG_m) \to H^2_{\mathrm{et}}(k(X),\bbG_m)$ is torsion-free, the canonical map \eqref{eq:canonical-1} is injective.
\end{corollary}
\begin{proof}
By construction, we have the following commutative square
$$
\xymatrix{
H^2_{\mathrm{et}}(X,\bbG_m) \simeq \mathrm{dBr}(X) \ar[d]_-{f^\ast} \ar[rr]^-{\eqref{eq:canonical-1}} && K_0^{(2)}(X) \ar[d]^-{f^\ast} \\
H^2_{\mathrm{et}}(k(X),\bbG_m) \simeq \mathrm{Br}(k(X)) \ar[rr]_-{\eqref{eq:canonical-1}} && K_0^{(2)}(k(X))\,,
}
$$
where $f\colon \mathrm{Spec}(k(X)) \to X$ stands for the canonical morphism. Hence, since by assumption the kernel of the left-hand side vertical homomorphism is torsion-free, the proof follows from the combination of Theorem \ref{thm:five} with the injectivity of the canonical map \eqref{eq:canonical-new} (with $k$ replaced by $k(X)$).
\end{proof}
Corollary \ref{cor:new} applies, in particular, to the following singular schemes:
\begin{example}[Cones over curves]\label{ex:curves}
Let $X$ be the affine cone over a smooth projective plane complex curve of degree $\geq 4$. By construction, $X$ has a single isolated singularity. As explained in  \cite[Pages~10-11]{Childs}\cite[Page~41]{Contemporary}, the kernel of the restriction homomorphism $H^2_{\mathrm{et}}(X,\bbG_m) \to H^2_{\mathrm{et}}(\bbC(X),\bbG_m)$ is torsion-free; this kernel is {\em not} finitely generated. Therefore, Corollary \ref{cor:new} implies that the canonical map \eqref{eq:canonical-1} is injective.
\end{example}
\begin{example}[Mumford's surface]\label{ex:Mumford1}
Let $C$ be a complex smooth cubic curve in $\bbP^2$ and $x_1, \ldots, x_{15}$ points of $C$ which are in general position except that on $C$ the divisor $\sum_i x_i \equiv 5 H$, where $H$ denotes an hyperplane section. Following \cite[Page 16]{Mumford}, blow-up every point $x_i$ and call $X'$ the resulting surface. The linear system of quintics through the $x_i$'s contracts the proper transform of $C$ and yields a normal projective surface $X$ with a single isolated singularity. As explained in \cite[Page~75]{Grothendieck2}, the kernel of the restriction homomorphism $H^2_{\mathrm{et}}(X,\bbG_m) \to H^2_{\mathrm{et}}(\bbC(X),\bbG_m)$ is torsion-free;  this kernel is {\em not} finitely generated. Therefore, Corollary \ref{cor:new} implies that the canonical map \eqref{eq:canonical-1} is injective.
%
\end{example}
\begin{remark}
Let $X$ be as in Corollary \ref{cor:new}. Note that since the kernel of the restriction homomorphism $H^2_{\mathrm{et}}(X,\bbG_m) \to H^2_{\mathrm{et}}(k(X),\bbG_m)$ is torsion-free, the restriction homomorphism $\mathrm{Br}(X)=\mathrm{Br}'(X) \to \mathrm{Br}(k(X))$ is injective. This applies, notably, to the above Examples \ref{ex:curves}-\ref{ex:Mumford1}.
\end{remark}

\section{Background on noncommutative motives}
For a book, resp. survey, on noncommutative motives consult \cite{book}, resp. \cite{survey}. Let $k$ be a base commutative ring and $R$ a commutative ring of coefficients. Recall from \cite[\S4.1]{book} the construction of the category of {\em noncommutative Chow motives} $\NChow(k)_R$. This category is $R$-linear, additive, idempotent complete, and rigid symmetric monoidal, Moreover, it comes equipped with a $\otimes$-functor $U(-)_R\colon \dgcat_{\mathrm{sp}}(k) \to \NChow(k)_R$ defined on the category of smooth proper dg categories. Furthermore, we have the following isomorphisms:
\begin{equation}\label{eq:iso-last}
 \Hom_{\NChow(k)_R}(U(\cA)_R,U(\cB)_R)\simeq K_0(\cD_c(\cA^\op \otimes^{\bf L} \cB))_R =: K_0(\cA^\op \otimes^{\bf L} \cB)_R\,.
\end{equation}
Under \eqref{eq:iso-last}, the composition law of $\NChow(k)_R$ corresponds to the derived tensor product of bimodules. In the case where $R=\bbZ$, we will write $\NChow(k)$ instead of $\NChow(k)_\bbZ$ and $U$ instead of $U(-)_\bbZ$.
\begin{remark}[Dg Azumaya algebras]\label{rk:Dg}
Let $A$ be a dg Azumaya $k$-algebra. Similarly to ordinary Azumaya $k$-algebras (see \cite[Lem.~8.10]{JIMJ}), we have the $\otimes$-equivalence of symmetric monoidal triangulated categories $\cD_c(k) \to \cD_c(A^\op \otimes^{\bf L} A), M \mapsto M\otimes^{\bf L} A$, where the symmetric monoidal structure of $\cD_c(k)$, resp. $\cD_c(A^\op \otimes^{\bf L} A)$, is given by $-\otimes^{\bf L}-$, resp. $-\otimes^{\bf L}_{A}-$. Consequently, we obtain an induced ring isomorphism
\begin{equation}\label{eq:ring-End}
\End_{\NChow(k)}(U(k)) \stackrel{\sim}{\too} \End_{\NChow(k)}(U(A))\,.
\end{equation}
\end{remark}
Given a rigid symmetric monoidal category $\cC$, its {\em $\cN$-ideal} is defined as follows
$$ \cN(a,b):=\{f \in \Hom_\cC(a,b)\,|\, \forall g \in \Hom_\cC(b,a)\,\,\mathrm{we}\,\,\mathrm{have}\,\,\mathrm{tr}(g\circ f)=0 \}\,,$$
where $\mathrm{tr}(g\circ f)$ denotes the categorical trace of the endomorphism $g\circ f$. The category of {\em noncommutative numerical motives} $\NNum(k)_R$ is defined as the idempotent completion of the quotient of $\NChow(k)_R$ by the $\otimes$-ideal $\cN$. By construction, this category  is $R$-linear, additive, idempotent complete, and rigid symmetric monoidal. 
\begin{notation}\label{not:CSA}
In the case where $k$ is a field, we will write $\mathrm{CSA}(k)_R$ for the full subcategory of $\NNum(k)_R$ consisting of the objects $U(A)_R$ with $A$ a central simple $k$-algebra, and $\mathrm{CSA}(k)^\oplus_R$ for the closure of $\mathrm{CSA}(k)_R$ under finite direct sums.
\end{notation}
\section{Proof of Theorem \ref{thm:main3}}\label{sec:proof1}
We start by proving that the canonical map \eqref{eq:canonical-new} is injective. As proved in \cite[Thm.~4.4]{ANT}, every short exact sequence of dg categories $0 \to \cA \to \cB \to \cC \to 0$, with $\cA$ smooth proper and $\cB$ proper, is necessarily split. This implies that the second Grothendieck ring $K_0^{(2)}(k)$ agrees with the Grothendieck ring~$\cP\cT(k)$ of smooth proper pretriangulated dg categories introduced in \cite[\S5]{BLL}; consult \cite[Thm.~1.1]{ANT} for details. Recall from \cite[Page 899]{ANT} that we have a ring homomorphism:
\begin{eqnarray*}
\cP\cT(k) \too K_0(\NChow(k)) && [\cA] \mapsto [U(\cA)]\,.
\end{eqnarray*}
By precomposing  it with the ring isomorphism $K_0^{(2)}(k)\simeq \cP\cT(k)$ and with \eqref{eq:canonical-new}, we hence obtain the following composition:
\begin{eqnarray}\label{eq:canonical222}
\mathrm{Br}(k) \stackrel{\eqref{eq:canonical-new}}{\too} K_0^{(2)}(k) \simeq \cP\cT(k) \too K_0(\NChow(k)) && [A]\mapsto [U(A)]\,.
\end{eqnarray}
In what follows, we prove that the composition \eqref{eq:canonical222} is injective; note that this implies that \eqref{eq:canonical-new} is also injective. We start by recalling the following results:
\begin{proposition}{(\cite[Prop.~6.2(i)]{ANT})}\label{prop:index1}
Let $k$ be a field, $A$ and $B$ two central simple $k$-algebras, and $R$ a commutative ring of positive prime characteristic\footnote{Recall that a commutative ring $R$ has characteristic zero, resp. positive prime characteristic $p>0$, if the kernel of the unique ring homomorphism $\bbZ \to R$ is $\{0\}$, resp. $p\bbZ$.} $p>0$. If $p\mid\mathrm{ind}(A^\op \otimes B)$, then $U(A)_R \not\simeq U(B)_R$ in $\NChow(k)_R$. Moreover, we have 
\begin{equation}\label{eq:vanishing}
\Hom_{\NNum(k)_R}(U(A)_R, U(B)_R)=\Hom_{\NNum(k)_R}(U(B)_R, U(A)_R)=0\,.
\end{equation}
\end{proposition}
\begin{proposition}{(\cite[Prop.~6.11]{ANT})}\label{prop:Br-graded1}
Let $k$ be a field and $R$ a field of positive characteristic $p>0$. In this case, the category $\mathrm{CSA}(k)^\oplus_R$ is equivalent to the category of $\mathrm{Br}(k)\{p\}$-graded finite dimensional $R$-vector spaces.
\end{proposition}
Let $A$ and $B$ be two central simple $k$-algebras such that $[A]\neq [B]$ in $\mathrm{Br}(k)$. We need to show that $[U(A)]\neq [U(B)]$ in the Grothendieck ring $K_0(\NChow(k))$. Recall that $\mathrm{ind}(A^\op \otimes B) =1$ if and only if $[A]=[B]$. Therefore, let us choose a prime number $p$ such that $p\mid \mathrm{ind}(A^\op \otimes B)$. Thanks to Proposition \ref{prop:index1}, we have $U(A)_{\bbF_p}\not\simeq U(B)_{\bbF_p}$ in $\NChow(k)_{\bbF_p}$. Note that since $\mathrm{End}_{\NChow(k)_{\bbF_p}}(U(k)_{\bbF_p})\simeq \bbF_p$, Remark \ref{rk:Dg} yields the following ring isomorphisms (in $\NChow(k)_{\bbF_p}$ and $\NNum(k)_{\bbF_p}$):
\begin{equation}\label{eq:identifications}
\mathrm{End}(U(A)_{\bbF_p}) \simeq \mathrm{End}(U(B)_{\bbF_p})\simeq \mathrm{End}(U(k)_{\bbF_p})\simeq \bbF_p\,.
\end{equation}
By combining \eqref{eq:identifications} with \eqref{eq:vanishing}, we conclude that $U(A)_{\bbF_p}\not\simeq U(B)_{\bbF_p}$ in $\NNum(k)_{\bbF_p}$. 

Now, let us assume by absurd that $[U(A)]=[U(B)]$ in the Grothendieck ring $K_0(\NChow(k))$. This is equivalent to the following condition:
\begin{equation}\label{eq:cond-last1}
\exists\, N\!\!M \in \NChow(k)\,\,\mathrm{such}\,\,\mathrm{that}\,\, U(A) \oplus N\!\!M \simeq U(B) \oplus N\!\!M\,.
\end{equation}
Thanks to Lemma \ref{lem:key11} below, if condition \eqref{eq:cond-last1} holds, then there exist integers $n, m \geq 0$ and a noncommutative numerical motive $N\!\!M'$~such~that 
\begin{equation}\label{eq:iso-global1}
\oplus^{n+1}_{i=1} U(A)_{\bbF_p} \oplus \oplus^m_{j=1} U(B)_{\bbF_p} \oplus N\!\!M' \simeq \oplus^n_{i=1} U(A)_{\bbF_p} \oplus \oplus^{m+1}_{j=1} U(B)_{\bbF_p} \oplus N\!\!M'
\end{equation}
in $\NNum(k)_{\bbF_p}$. Moreover, $N\!\!M'$ does not contains $U(A)_{\bbF_p}$ neither $U(B)_{\bbF_p}$ as direct summands. Note that the composition bilinear pairing (in $\NNum(k)_{\bbF_p}$)
\begin{equation}\label{eq:pairing-11}
\Hom(U(A)_{\bbF_p}, N\!\!M') \times \Hom(N\!\!M', U(A)_{\bbF_p}) \too \End(U(A)_{\bbF_p})
\end{equation}
is zero; similarly with $U(A)_{\bbF_p}$ replaced by $U(B)_{\bbF_p}$. This follows from the fact that the right-hand side~of \eqref{eq:pairing-11} is isomorphic to $\bbF_p$, from the fact that the category $\NNum(k)_{\bbF_p}$ is $\bbF_p$-linear, and from the fact that the noncommutative numerical motive $N\!\!M'$ does not contains $U(A)_{\bbF_p}$ as a direct summand. The following composition bilinear pairing (in $\NNum(k)_{\bbF_p}$)
\begin{equation}\label{eq:pairing-22}
\Hom(U(A)_{\bbF_p}, N\!\!M')\times \Hom(N\!\!M', U(B)_{\bbF_p})\too \Hom(U(A)_{\bbF_p}, U(B)_{\bbF_p})
\end{equation}
is also zero; similarly with $U(A)_{\bbF_p}$ and $U(B)_{\bbF_p}$ replaced by $U(B)_{\bbF_p}$ and $U(A)_{\bbF_p}$, respectively. This follows from the fact that the right-hand side of \eqref{eq:pairing-22} is zero; see Proposition \ref{prop:index1}. Note that the triviality of the composition bilinear pairings \eqref{eq:pairing-11}-\eqref{eq:pairing-22} implies that the above isomorphism \eqref{eq:iso-global1} restricts to an isomorphism
$$
U(A)_{\bbF_p}\oplus \oplus^n_{i=1} U(A)_{\bbF_p} \oplus \oplus^m_{j=1} U(B)_{\bbF_p} \simeq U(B)_{\bbF_p}\oplus \oplus^n_{i=1} U(A)_{\bbF_p} \oplus \oplus^m_{j=1} U(B)_{\bbF_p}
$$
in the subcategory $\mathrm{CSA}(k)_{\bbF_p}^\oplus \subset \NNum(k)_{\bbF_p}$. Since $\mathrm{CSA}(k)_{\bbF_p}^\oplus$ is equivalent to the category of $\mathrm{Br}(k)\{p\}$-graded finite dimensional $\bbF_p$-vector spaces (see Proposition \ref{prop:Br-graded1}), it satisfies the cancellation property with respect to direct sums. Consequently, we conclude from the preceding isomorphism that $U(A)_{\bbF_p}\simeq U(B)_{\bbF_p}$ in $\NNum(k)_{\bbF_p}$, which is a contradiction. This shows that the above composition \eqref{eq:canonical222} is injective.
\begin{lemma}\label{lem:key11}
If the above condition \eqref{eq:cond-last1} holds, then there exist integers $n,m \geq 0$ and a noncommutative numerical motive $N\!\!M' \in \NNum(k)_{\bbF_p}$ such that:
\begin{itemize}
\item[(i)] We have $N\!\!M_{\bbF_p}\simeq \oplus^n_{i=1} U(A)_{\bbF_p} \oplus \oplus _{j=1}^m U(B)_{\bbF_p} \oplus N\!\!M'$ in $\NNum(k)_{\bbF_p}$.
\item[(ii)] The noncommutative numerical motive $N\!\!M'$ does not contains $U(A)_{\bbF_p}$ neither $U(B)_{\bbF_p}$ as direct summands.
\end{itemize}
\end{lemma}
\begin{proof}
Recall that the category $\NNum(k)_{\bbF_p}$ is idempotent complete. Therefore, by inductively splitting the (possible) direct summands $U(A)_{\bbF_p}$ and $U(B)_{\bbF_p}$ of the noncommutative numerical motive $N\!\!M_{\bbF_p}$, we obtain an isomorphism
$$ N\!\!M_{\bbF_p} \simeq U(A)_{\bbF_p} \oplus \cdots \oplus U(A)_{\bbF_p} \oplus U(B)_{\bbF_p} \oplus \cdots \oplus U(B)_{\bbF_p} \oplus N\!\!M'$$
in $\NNum(k)_{\bbF_p}$, with $N\!\!M'$ satisfying condition (ii). We claim that the number of copies of $U(A)_{\bbF_p}$ and $U(B)_{\bbF_p}$ is finite; note that this concludes the proof. We will focus ourselves in the case $U(A)_{\bbF_p}$; the proof of the case $U(B)_{\bbF_p}$ is similar. Suppose by absurd that the number of copies of $U(A)_{\bbF_p}$ is infinite. In this case, the natural (shift) isomorphism $U(A)_{\bbF_p} \oplus \oplus^\infty_{i=1} U(A)_{\bbF_p} \simeq \oplus^\infty_{i=1} U(A)_{\bbF_p}$ would allows us to conclude that $[U(A)_{\bbF_p}]=0$ in $K_0(\NNum(k)_{\bbF_p})$. By construction, the category $\NNum(k)_{\bbF_p}$ is additive and rigid symmetric monoidal. Therefore, the classical Euler characteristic construction\footnote{Let $(\cC,\otimes, {\bf 1})$ be a rigid symmetric monoidal category. Given a dualizable object $a \in \cC$, recall that its {\em Euler characteristic} $\chi(a)$ is defined as ${\bf 1} \stackrel{\mathrm{co}}{\to} a^\vee \otimes a \simeq a \otimes a^\vee \stackrel{\mathrm{ev}}{\to} {\bf 1}$, where $\mathrm{co}$, resp. $\mathrm{ev}$, denotes the coevaluation, resp. evaluation, morphism.} gives rise to a ring homomorphism:
\begin{equation*}
\chi\colon K_0(\NNum(k)_{\bbF_p}) \too \mathrm{End}_{\NNum(k)_{\bbF_p}}(U(k)_{\bbF_p})\,.
\end{equation*}
Consequently, since $[U(A)_{\bbF_p}]=0$, we would conclude that $\chi([U(A)_{\bbF_p}])=0$. Thanks to the above ring isomorphisms \eqref{eq:identifications}, the equality $\chi([U(A)_{\bbF_p}])=\chi([U(k)_{\bbF_p}])$ holds in $\bbF_p$. Using the fact that $\chi([U(k)_{\bbF_p}])=1 \in \bbF_p$, we hence obtain a contradiction. This finishes the proof.
\end{proof}

We now have all the ingredients necessary for the conclusion of the proof of Theorem \ref{thm:main3}. Since by assumption the quasi-compact quasi-separated $k$-scheme $X$ is regular and integral, we can consider the canonical morphism $f\colon \mathrm{Spec}(k(X))\to X$ and the associated commutative square:
\begin{equation*}\label{eq:square-last}
\xymatrix{
H^2_{\mathrm{et}}(X,\bbG_m)\simeq \mathrm{dBr}(X) \ar[rr]^-{\eqref{eq:canonical-1}} \ar[d]_-{f^\ast} && K_0^{(2)}(X) \ar[d]^-{f^\ast} \\
H^2_{\mathrm{et}}(k(X),\bbG_m)_{\mathrm{tor}}\simeq \mathrm{Br}(k(X)) \ar[rr]_-{\eqref{eq:canonical-1}} && K_0^{(2)}(k(X))\,.
}
\end{equation*}
As proved in \cite[Cor.~1.10]{Grothendieck2}, the left-hand side vertical homomorphism is injective. Therefore, we conclude that the injectivity of \eqref{eq:canonical-1} follows from the injectivity of the canonical map \eqref{eq:canonical-new} (with $k$ replaced by $k(X)$).

\section{Proof of Theorem \ref{thm:five}}
The proof is divided into two steps: in the first step we consider the ``local'' case and in the second step the ``global'' case.
\subsection*{Step I}
Let $k$ be a commutative ring. In the particular case where $X=\mathrm{Spec}(k)$, let us write the canonical map \eqref{eq:canonical-1} as follows:
\begin{equation}\label{eq:canonical1}
\mathrm{dBr}(k) \too K_0^{(2)}(k)\,.
\end{equation}
The following ``local'' result is of independent interest:
\begin{theorem}\label{thm:main22}
Assume that $K_0(k)_\bbQ \simeq \bbQ$. Given dg Azumaya $k$-algebras $A$ and $B$ such that $A^\op \otimes^{\bf L}B$ is not Morita equivalent to an ordinary Azumaya $k$-algebra, the images of $[A]$ and $[B]$ under the canonical map \eqref{eq:canonical1} are different.
\end{theorem}
Recall that $K_0(k)\simeq \bbZ$ whenever $k$ is a local ring or a principal ideal domain.
\begin{corollary}
The restriction of the canonical map \eqref{eq:canonical1} to the torsion-free subgroup $H^1_{\mathrm{et}}(\mathrm{Spec}(k),\bbZ)$ is injective.
\end{corollary}
\begin{example}[Nodal curve]
Let $F$ be a field and $k:=(F[x,y]/(y^2-x^3 +x^2))_{(x,y)}$ the local ring of the singular point of the nodal curve. As explained in \cite[Example~2.2]{Pic-Weibel}, we have $H^1_{\mathrm{et}}(\mathrm{Spec}(k),\bbZ)\simeq \bbZ$. 
\end{example}
\begin{corollary}\label{cor:last222}
Given a dg Azumaya $k$-algebra $A$ which is not Morita equivalent to an ordinary Azumaya $k$-algebra, the restriction of the canonical map \eqref{eq:canonical1} to the subgroup $\bbZ$ of $\mathrm{dBr}(k)$ generated by $[A]$ is injective.
\end{corollary}
\begin{corollary}\label{cor:last1}
Given a dg Azumaya $k$-algebra $A$ which is not Morita equivalent to an ordinary Azumaya $k$-algebra, the image of $[A]$ under the canonical map \eqref{eq:canonical1} is different from the images of the ordinary Azumaya $k$-algebras.
\end{corollary}
\subsection*{Proof of Theorem \ref{thm:main22}}
As in proof of Theorem \ref{thm:main3}, we have the composition:
\begin{eqnarray}\label{eq:canonical2222}
\mathrm{dBr}(k) \stackrel{\eqref{eq:canonical1}}{\too} K_0^{(2)}(k) \simeq \cP\cT(k) \too K_0(\NChow(k)) && [A]\mapsto [U(A)]\,.
\end{eqnarray}
In what follows, we prove that $[U(A)]\neq [U(B)]$ in $K_0(\NChow(k))$; note that this automatically implies Theorem \ref{thm:main22}. We start by recalling the following result:
\begin{theorem}(\cite[Thm.~B.15]{Separable})\label{thm:last}
Given a dg Azumaya $k$-algebra $A$ which is not Morita equivalent to an ordinary Azumaya $k$-algebra, we have $U(A)_\bbQ \not\simeq U(k)_\bbQ$ in $\NChow(k)_\bbQ$.
\end{theorem}
\begin{proposition}\label{prop:new11}
Assume that $K_0(k)_\bbQ\simeq \bbQ$. Given dg Azumaya $k$-algebras $A$ and $B$ such that $A^\op \otimes^{\bf L} B$ is not Morita equivalent to an ordinary Azumaya $k$-algebra, we have $U(A)_\bbQ\not\simeq U(B)_\bbQ$ in $\NNum(k)_\bbQ$. Moreover, we have
\begin{equation}\label{eq:vanishing1}
\Hom_{\NNum(k)_\bbQ}(U(A)_\bbQ, U(B)_\bbQ) = \Hom_{\NNum(k)_\bbQ}(U(B)_\bbQ, U(A)_\bbQ) =0\,.
\end{equation}
\end{proposition}
\begin{proof}
Since $A^\op \otimes^{\bf L} B$ is not Morita equivalent to an ordinary Azumaya $k$-algebra, Theorem \ref{thm:last} (with $A$ replaced by $A^\op \otimes^{\bf L}B$) implies that $U(A^\op \otimes^{\bf L}B)_\bbQ \not\simeq U(k)_\bbQ$ in $\NChow(k)_\bbQ$. Consequently, since the functor $U(-)_\bbQ$ is symmetric monoidal and $A^\op \otimes^{\bf L} A$ is Morita equivalent to $k$, we have $U(A)_\bbQ \not\simeq U(B)_\bbQ$ in $\NChow(k)_\bbQ$. By assumption, $\mathrm{End}_{\NChow(k)_\bbQ}(U(k)_\bbQ)\simeq K_0(k)_\bbQ \simeq \bbQ$. Making use of Remark \ref{rk:Dg}, we hence obtain ring isomorphisms (in $\NChow(k)_\bbQ$ and $\NNum(k)_\bbQ$):
\begin{equation}\label{eq:isos-new}
\mathrm{End}(U(A)_\bbQ) \simeq \mathrm{End}(U(B)_\bbQ)\simeq \mathrm{End}(U(k)_\bbQ)\simeq \bbQ\,.
\end{equation}
Note that the composition bilinear pairing (in $\NChow(k)_\bbQ$)
\begin{equation}\label{eq:pairing-new1}
\Hom(U(A)_\bbQ, U(B)_\bbQ) \times \Hom(U(B)_\bbQ, U(A)_\bbQ) \too \mathrm{End}(U(A)_\bbQ)
\end{equation}
is zero; similarly with $U(A)_\bbQ$ and $U(B)_\bbQ$ replaced by $U(B)_\bbQ$ and $U(A)_\bbQ$, respectively. This follows from the fact that the right-hand side of \eqref{eq:pairing-new1} is isomorphic to $\bbQ$, from the fact that the category $\NChow(k)_\bbQ$ is $\bbQ$-linear, and from the fact that $U(A)_\bbQ \not \simeq U(B)_\bbQ$ in $\NChow(k)_\bbQ$. Making use of the ring isomorphisms \eqref{eq:isos-new}, we hence conclude not only that $U(A)_\bbQ \not\simeq U(B)_\bbQ$ in $\NNum(k)_\bbQ$ but also that the $\bbQ$-vector spaces \eqref{eq:vanishing1} are zero.
\end{proof}
Now, let us assume by absurd that $[U(A)]=[U(B)]$ in the Grothendieck ring $K_0(\NChow(k))$. This is equivalent to the following condition:
\begin{equation}\label{eq:cond-last11}
\exists\, N\!\!M \in \NChow(k)\,\,\mathrm{such}\,\,\mathrm{that}\,\, U(A) \oplus N\!\!M \simeq U(B) \oplus N\!\!M\,.
\end{equation}
Thanks to Lemma \ref{lem:key111} below, if condition \eqref{eq:cond-last11} holds, then there exist integers $n, m \geq 0$ and a noncommutative numerical motive $N\!\!M'$~such~that 
\begin{equation}\label{eq:iso-global-11}
\oplus^{n+1}_{i=1} U(A)_\bbQ \oplus \oplus^m_{j=1} U(B)_\bbQ \oplus N\!\!M' \simeq \oplus^n_{i=1} U(A)_\bbQ \oplus \oplus^{m+1}_{j=1} U(B)_\bbQ \oplus N\!\!M'
\end{equation}
in $\NNum(k)_\bbQ$. Moreover, $N\!\!M'$ does not contains $U(A)_\bbQ$ neither $U(B)_\bbQ$ as direct summands. Similarly to the proof of Theorem \ref{thm:main3} (with $\bbF_p$ replaced by $\bbQ$ and with Proposition \ref{prop:index1} replaced by Proposition \ref{prop:new11}), we conclude that the preceding isomorphism \eqref{eq:iso-global-11} restricts to an isomorphism
$$
U(A)_\bbQ\oplus \oplus^n_{i=1} U(A)_\bbQ \oplus \oplus^m_{j=1} U(B)_\bbQ \simeq U(B)_\bbQ \oplus \oplus^n_{i=1} U(A)_\bbQ \oplus \oplus^m_{j=1} U(B)_\bbQ\,.
$$
Let us denote by $\langle U(A)_\bbQ, U(B)_\bbQ\rangle^\oplus$ the smallest additive and idempotent complete full subcategory of $\NNum(k)_\bbQ$  containing the objects $U(A)_\bbQ$ and $U(B)_\bbQ$. Note that the preceding isomorphism belongs to this subcategory. It is well-known that the functor $\Hom_{\NNum(k)_\bbQ}(U(A)_\bbQ\oplus U(B)_\bbQ, -)$ induces an equivalence between $\langle U(A)_\bbQ, U(B)_\bbQ\rangle^\oplus$ and the category of finitely generated projective (right) modules over the ring $\mathrm{End}_{\NNum(k)_\bbQ}(U(A)_\bbQ\oplus U(B)_\bbQ)$. Thanks to Proposition \ref{prop:new11}, and to the ring isomorphisms \eqref{eq:isos-new}, the latter ring is isomorphic to $\bbQ\times \bbQ$. Therefore, the category $\langle U(A)_\bbQ, U(B)_\bbQ\rangle^\oplus$ is equivalent to the product $\mathrm{Vect}(\bbQ) \times \mathrm{Vect}(\bbQ)$, where $\mathrm{Vect}(\bbQ)$ stands for the category of finite dimensional $\bbQ$-vector spaces. In particular, the latter category satisfies the cancellation property with respect to direct sums. We hence conclude from the preceding isomorphism that $U(A)_\bbQ \simeq U(B)_\bbQ$ in $\NNum(k)_\bbQ$, which is a contradiction. This shows that $[U(A)]\neq [U(B)]$ in $K_0(\NChow(k))$, and therefore concludes the proof of Theorem \ref{thm:main22}.
\begin{lemma}\label{lem:key111}
If the above condition \eqref{eq:cond-last11} holds, then there exist integers $n,m \geq 0$ and a noncommutative numerical motive $N\!\!M' \in \NNum(k)_\bbQ$ such that:
\begin{itemize}
\item[(i)] We have $N\!\!M_\bbQ \simeq \oplus^n_{i=1} U(A)_\bbQ \oplus \oplus _{j=1}^m U(B)_\bbQ \oplus N\!\!M'$ in $\NNum(k)_\bbQ$.
\item[(ii)] The noncommutative numerical motive $N\!\!M'$ does not contains $U(A)_\bbQ$ neither $U(B)_\bbQ$ as direct summands.
\end{itemize}
\end{lemma}
\begin{proof}
The proof is similar to the proof of Lemma \ref{lem:key11}; simply replace $\bbF_p$ by $\bbQ$ and the ring isomorphisms \eqref{eq:identifications} by the ring isomorphisms \eqref{eq:isos-new}.
\end{proof}
\subsection*{Step II}
Let $x$ be the isolated singularity of $X$. Since by assumption $X$ is normal, the affine $k$-scheme $\mathrm{Spec}(\cO_{X,x})$ is also normal. Consequently, we have $H^1_{\mathrm{et}}(X,\bbZ)=H^1_{\mathrm{et}}(\cO_{X,x},\bbZ)=0$. Consider the canonical morphism $f\colon \mathrm{Spec}(\cO_{X,x}) \to X$ and the associated commutative square:
\begin{equation}\label{eq:2-squares}
\xymatrix{
H^2_{\mathrm{et}}(X,\bbG_m) \simeq \mathrm{dBr}(X) \ar[d]_-{f^\ast} \ar[rr]^-{\eqref{eq:canonical-1}} &&    K_0^{(2)}(X) \ar[d]^-{f^\ast} \\
H^2_{\mathrm{et}}(\cO_{X,x},\bbG_m) \simeq \mathrm{dBr}(\cO_{X,x}) \ar[rr]_-{\eqref{eq:canonical-1}} && K_0^{(2)}(\cO_{X,x})\,.
}
\end{equation}
As proved in \cite[Thm.~4.1]{OPS} (see also \cite[Lem.~3]{Bertuccioni}), the left-hand side vertical homomorphism in \eqref{eq:2-squares} is injective. Note that since by assumption $A^\op \otimes^{\bf L} B$ is not Morita equivalent to a sheaf of ordinary Azumaya algebras, this injectivity implies that $f^\ast(A^\op \otimes^{\bf L} B) \simeq f^\ast(A)^\op \otimes^{\bf L} f(B)$ is not Morita equivalent to an ordinary Azumaya $\cO_{X,x}$-algebra; otherwise $f^\ast(A^\op \otimes^{\bf L} B)$ would be a torsion class because $\mathrm{Br}(X)\simeq H^2_{\mathrm{et}}(X,\bbG_m)_{\mathrm{tor}}$ and $\mathrm{Br}(\cO_{X,x}) \simeq H^2_{\mathrm{et}}(\cO_{X,x}, \bbG_m)_{\mathrm{tor}}$. Making use of the commutative square \eqref{eq:2-squares} and of Theorem \ref{thm:main22} (with $k$ replaced by the local ring $\cO_{X,x}$ and with $A$ and $B$ replaced by $f^\ast(A)$ and $f^\ast(B)$, respectively), we hence conclude that the images of $[A]$ and $[B]$ under the canonical map \eqref{eq:canonical-1} are different. This finishes the proof of Theorem \ref{thm:five}.
\begin{remark}[Injectivity]
The above proof of Step II shows that whenever the canonical map \eqref{eq:canonical1} (with $k$ replaced by the local ring $\cO_{X,x}$) is injective, the canonical map \eqref{eq:canonical-1} is also injective.
\end{remark}

\end{document}

\end{proof}